\documentclass{article}

\usepackage{amsmath,amssymb,amsfonts,amsthm,graphicx, float, lineno, hyperref, authblk} %
\usepackage[english]{babel}
\newtheorem{theorem}{Theorem}[section]
\newtheorem{lemma}[theorem]{Lemma}
\newtheorem{proposition}[theorem]{Proposition}

\newtheorem{definition}[theorem]{Definition}
\newtheorem{remark}[theorem]{Remark}

\newcommand{\E}{\mathbb{E}}
\newcommand{\N}{\mathbb{N}}
\newcommand{\Var}{\mathrm{Var}}
\newcommand{\Cov}{\mathrm{Cov}}

\numberwithin{equation}{section}

\title{Limit Theorems for $\theta$-expansions and the Failure of the Strong Law}

\author[1]{Andreas Rusu}
\affil[1]{Faculty of Applied Sciences, National University of Science and Technology POLITEHNICA Bucharest, Splaiul Independentei 313, 060042 Bucharest, Romania\\ 
e-mail: \href{mailto: andreasrusu10@gmail.com}{andreasrusu10@gmail.com}}

\author[2,3]{Gabriela Ileana Sebe}
\affil[2]{Faculty of Applied Sciences, National University of Science and Technology POLITEHNICA Bucharest, Splaiul Independentei 313, 060042 Bucharest, Romania }
\affil[3]{Gheorghe Mihoc-Caius Iacob Institute of Mathematical Statistics and Applied Mathematics of the Romanian Academy, Calea 13 Sept. 13, 050711 Bucharest, Romania\\ e-mail: \href{mailto: igsebe@yahoo.com}{igsebe@yahoo.com}}

\author[4]{Dan Lascu}
\affil[4]{Romanian Naval Academy ``Mircea cel Batran", 1 Fulgerului, 900218 Constanta, Romania\\ 
e-mail: \href{mailto: lascudan@gmail.com}{lascudan@gmail.com}}

\date{September 2025}

\begin{document}

\maketitle
\begin{abstract}
The paper presents fundamental metrical theorems for a class of continued fraction-like expansions known as $\theta$-expansions. We first prove Khinchine's Weak Law of Large Numbers for the sum of digits, followed by the Diamond-Vaaler Strong Law for the sum of digits minus the largest one. 
Our main result is a general theorem on the failure of the strong law, showing that no regular norming sequence can yield a finite, non-zero almost sure limit. This result extends a classical theorem of Philipp to the $\theta$-expansion setting. The proofs leverage the system's explicit invariant measure and a detailed analysis of its mixing properties.
\end{abstract}



\section{Introduction and Setup}

The metrical theory of continued fractions is a cornerstone of ergodic theory and metric number theory, offering profound insights into the distribution of digits and the behavior of sums involving them. 
Classical results by Khinchine, L\'evy, and others describe the asymptotic growth of the sum of partial quotients.
A central finding is that while the mean of the digits is infinite, a weak law of large numbers holds. 
Furthermore, Diamond and Vaaler \cite{diamondvaaler} showed that the strong law fails for the full sum but holds for the sum minus the largest term, underscoring the outsized influence of a single, rare large digit.

This paper extends this rich theory of regular continued fractions to a class of generalized continued fraction algorithms known as 
$\theta$-expansions. 
These algorithms are generated by a family of piecewise monotonic interval maps $T_{\theta}$, which include the standard Gauss map as a special case. A significant challenge in this generalization is the absence of a closed-form invariant measure for arbitrary $\theta$. 
We overcome this by focusing on a specific class of parameters for which an explicit and finite invariant measure exists.

\subsection{$\theta$-Expansions}

Let $\theta \in (0,1)$ be an irrational number such that $\theta^2 = \frac{1}{m}$ for some $m \in \mathbb{N}_+$. This restriction ensures the existence of a particularly well-behaved invariant measure.
Also, to ensure $\theta$ is irrational we assume $m$ is not a perfect square.

\begin{definition}
The {generalized Gauss map} $T_{\theta}: [0, \theta] \to [0, \theta]$ is defined by
\[
T_{\theta}(x) =
\begin{cases}
\displaystyle \frac{1}{x} - \theta \left\lfloor \frac{1}{\theta x} \right\rfloor & \text{for } x \in (0, \theta], \\
0 & \text{for } x = 0.
\end{cases}
\]
\end{definition}
\noindent 
This map generates a unique {$\theta$-expansion} for any irrational $x \in (0, \theta]$:
\[
x = \cfrac{1}{\theta\ell_1(x) + \cfrac{1}{\theta\ell_2(x) + \cfrac{1}{\theta\ell_3(x) + \ddots}}} := [\ell_1 (x), \ell_2 (x), \ell_3 (x), \ldots]_\theta,
\]
where the digits are given by
\[
\ell_n (x) = \ell_{n, \theta} (x)= \left\lfloor \frac{1}{\theta T_{\theta}^{n-1}(x)} \right\rfloor, \quad n \geq 1,
\]
with $T_{\theta}^0=x$. 
The digits take values only in $\{m,m+1,\dots\}$, i.e., $\ell_n\ge m$ for every $n \geq 1$.

The map $T_\theta$ partitions the interval $[0, \theta]$ into the subintervals
\begin{equation} \label{cylinder}
I_\theta(i) = \left( \frac{1}{\theta(i+1)}, \frac{1}{\theta i} \right],
\end{equation}
on which $T_{\theta}(x) = x^{-1} - i \theta$.

The ergodic and statistical properties of this system are governed by an explicit invariant probability measure.

\begin{definition}[Invariant Measure] \label{inv.measure}
The {$\theta$-Gauss measure} $\gamma_{\theta}$ on $[0, \theta]$ is defined by
\[
\frac{\mathrm{d}\gamma_{\theta}}{\mathrm{d}x}(x) = h_\theta(x) = \frac{1}{\log(1+\theta^2)} \cdot \frac{\theta}{1 + \theta x}.
\]
That is, for any measurable set $A \subset [0, \theta]$,
\[
\gamma_{\theta}(A) = \frac{1}{\log(1+\theta^2)} \int_A \frac{\theta}{1+\theta x}  \mathrm{d} x.
\]
The system $(T_{\theta}, \gamma_{\theta})$ is known to be ergodic \cite{CR-2003}.
\end{definition}

The metrical theory of 
$\theta$-expansions has been developed in several directions. Building on the ergodicity of the system \cite{CR-2003}, Sebe and Lascu established fine convergence rate estimates for the Gauss-Kuzmin problem \cite{SebeLascuCraiova,Sebe2017,SebeLascu2019}. 
Further work has extended these investigations to the extremes of the digit sequence, yielding results in extreme value theory \cite{SebeLascuBilel3}, and to the study of exceptional sets through the lens of Hausdorff dimension \cite{SebeLascuBilel1, SebeLascuBilel2}. The present paper complements this body of work by providing the fundamental limit theorems governing the growth of the digit sum.

For $n \geq 1$, we define the fundamental objects of our study: 
\[
S_{n,\theta}(x) := \displaystyle \sum_{k=1}^n \ell_k (x), \quad 
L_{n,\theta}(x) := \max_{1 \leq k \leq n} {\ell_k(x)}
\]
the sum of the first $n$ digits and the largest digit, respectively.    

The divergence of the mean, $\E_{\gamma_{\theta}}[\ell_1] = \infty$ (as shown in Remark \ref{rem.1.6}), places these objects in the domain of applications of ergodic theory to non-integrable functions. The purpose of this paper is to establish the following limit theorems governing their growth.

\begin{theorem}[Khinchine’s Theorem for $\theta$-Expansions]
\label{thm:khinchine}
Let $\theta \in (0, 1)$ such that $\theta^2 = 1/m$ for some $m \in \N_+$. 
For every $\varepsilon>0$,
\[
\lim_{n\to\infty}
\gamma_\theta\!\left(
\left\{x\in[0,\theta]:
\left|
\frac{S_{n,\theta}(x)}{n\log n}
- \frac{1}{\log(1+\theta^2)}
\right|>\varepsilon
\right\}\right)=0.
\]
\end{theorem}

\begin{theorem}[Diamond--Vaaler Theorem for $\theta$-Expansions]
\label{thm:diamondvaaler}
Let $\theta \in (0, 1)$ such that $\theta^2 = 1/m$ for some $m \in \N_+$. 
For $\gamma_\theta$-almost every $x\in[0,\theta]$,
\[
\lim_{n\to\infty}
\frac{S_{n,\theta}(x)-L_{n,\theta}(x)}{n\log n}
= \frac{1}{\log(1+\theta^2)}.
\]
\end{theorem}

\begin{remark}
Theorem~\ref{thm:khinchine} generalizes Khinchine’s classical result for continued fractions, while Theorem~\ref{thm:diamondvaaler} extends the Diamond--Vaaler law describing the negligible role of the maximal partial quotient.  
\end{remark}
\begin{remark} \label{rem.1.6}
The need for these theorems stems from the divergence of the mean:
\begin{align*}
\E_{\gamma_{\theta}}[\ell_1] &= \int_0^\theta \ell_1 (x)\mathrm{d}\gamma_{\theta}(x)  
= \sum_{i=m}^\infty i \cdot \gamma_\theta(I_\theta(i)) \\
&= \sum_{i=m}^\infty i \cdot \frac{1}{\log(1+\theta^2)} \log\left( \frac{(i+1)^2}{i(i+2)} \right) = \sum_{i=m}^\infty i \cdot \frac{1}{\log(1+\theta^2)} \log\left( 1+\frac{1}{i(i+2)} \right).
\end{align*}
Since $\log (1+x) \geq x/2$, for $0 \leq x \leq 1$, we have
\[
\log\left( 1+\frac{1}{i(i+2)} \right) \geq \frac{1}{2i(i+2)}.
\]
Therefore, 
\[
i \cdot \log\left( 1+\frac{1}{i(i+2)} \right) \geq \frac{1}{2(i+2)}.
\]
Since the series $\sum_{i=m}^\infty \frac{1}{2(i+2)}$ diverges, by the comparison test, $\sum_{i=m}^\infty i \cdot \gamma_\theta(I_\theta(i))$ diverges, 
i.e., the digit function $\ell_1(x)$ is not $\gamma_\theta$-integrable.
\end{remark}
\begin{theorem}[Failure of the Strong Law for $\theta$-Expansions]\label{thm:philipp}
Let $\{a(n)\}_{n \geq 1}$ be a sequence of positive numbers such that $a(n)/n$ is non-decreasing. Then, for $\gamma_\theta$-almost every $x \in [0, \theta]$,
\[
\lim_{n \to \infty} \frac{S_{n,\theta}(x)}{a(n)} = 0 \quad \text{or} \quad \limsup_{n \to \infty} \frac{S_{n,\theta}(x)}{a(n)} = \infty
\]
according as
\[
\sum_{n=1}^\infty \frac{1}{a(n)} < \infty \quad \text{or} \quad = \infty.
\]
\end{theorem}

\begin{remark}
The condition that $a(n)/n$ is non-decreasing defines the class of ``regular norming sequences'' for which the theorem applies. This includes common choices like $a(n) = n\log n$, $a(n) = n\log\log n$, $a(n) = n^p$ for $p>1$ etc., while excluding pathological sequences with irregular growth.
\end{remark}
\begin{remark}
For the divergence part ($\sum 1/a(n) = \infty$), no regularity condition on $a(n)$ is needed.
\end{remark}

The final result generalizes a classical theorem of Philipp \cite{Philipp}  for regular continued fractions and confirms that the non-integrability of the digits is too severe to be remedied by any regular scaling. The proofs rely on the precise asymptotics of the digit distribution and, crucially, on establishing that the dynamical system is $\psi$-mixing (Proposition \ref{eq:psimixing}), a strong statistical property that allows us to apply powerful limit theorems for weakly dependent sequences.

The paper is structured as follows: Section \ref{sec.preliminaries} covers preliminaries, including the digit distribution and the proof of $\psi$-mixing. 
Sections \ref{sect.Khinchine}, \ref{sec:proof-diamondvaaler}, and \ref{sect.philipp} are dedicated to the proofs of Theorems \ref{thm:khinchine}, \ref{thm:diamondvaaler}, and \ref{thm:philipp}, respectively.

\section{Preliminaries} \label{sec.preliminaries}

The proofs of our main results rely on two foundational pillars: a precise understanding of the distribution of the digits $\ell_n(x)$ and strong statistical independence properties of the sequence $\{\ell_n\}$ under $\gamma_{\theta}$.
Throughout the paper, we will consider 
$\theta \in (0, 1)$ such that $\theta^2 = 1/m$ for some $m \in \N_+$.

\subsection{Distribution of the digits}

The following lemma provides the exact tail probability and asymptotic behavior of the digits, which is fundamental for all subsequent analysis.

\begin{lemma} \label{lema.2.1}
Let $m\in\mathbb N_+$ be given by $\theta^2=1/m$. For every integer $k \ge m$,
\[
\gamma_\theta(\ell_1\ge k)
=\frac{1}{\log(1+\theta^2)}\,
\log\!\Big(1+\frac{1}{k}\Big).
\]
In particular, as $k\to\infty$,
\[
\gamma_\theta(\ell_1\ge k)\sim\frac{1}{\log(1+\theta^2)}\cdot\frac{1}{k}.
\]
\end{lemma}

\begin{proof}
Recall that $\{x: \ell_1(x) = i\} = I_\theta(i)$, where $I_\theta(i)$ is given by (\ref{cylinder}). Therefore, for any $k > m$, we have:
\begin{align*}
\gamma_\theta(\ell_1\ge k) &= \sum_{i \geq k} \gamma_\theta(I_\theta(i)) = \sum_{i \geq k} \int_{1/(\theta(i+1))}^{1/(\theta i)} \frac{1}{\log(1+\theta^2)} \cdot \frac{\theta}{1 + \theta x}  dx \\
&= \frac{1}{\log(1+\theta^2)} \sum_{i \geq k} \left[ \log(1 + \theta x) \right]_{x=1/(\theta(i+1))}^{x=1/(\theta i)} \\
&= \frac{1}{\log(1+\theta^2)} \sum_{i \geq k} 
\left[ \log \left( 1+\frac{1}{i} \right) - 
\log \left( 1+\frac{1}{i+1} \right) \right] \\
&= \frac{1}{\log(1+\theta^2)} \log \left( 1+\frac{1}{k} \right).
\end{align*}
The asymptotic \(\gamma_\theta(\ell_1\ge k)\sim (1/\log(1+\theta^2))\cdot 1/k\)
is immediate from the expansion \(\log(1+1/k)\sim 1/k\) as \(k\to\infty\).
\end{proof}

\begin{remark}
By the $T_\theta$-invariance of $\gamma_\theta$, the same asymptotic holds for $\ell_n(x)$ for any $n \ge 1$, i.e, 
$\gamma_\theta(\ell_n\ge k)
=\frac{1}{\log(1+\theta^2)}\,
\log\!\Big(1+\frac{1}{k}\Big)$. 
\end{remark}
\subsection{$\psi$-mixing property}

The proof of Theorem \ref{thm:philipp} requires a strong control on the dependence between distant digits. This is guaranteed by the following proposition.

\begin{proposition}[$\psi$-Mixing]\label{prop:mixing}
Let $\theta \in (0, 1)$ such that $\theta^2 = 1/m$ for some $m \in \N_+$. 
The dynamical system $([0, \theta], \mathcal{B}, \gamma_\theta, T_\theta)$ is $\psi$-mixing. That is, there exist constants $K >0$ and $0 < \rho < 1$ such that for any $A \in \sigma(\ell_1, \ldots, \ell_k)$ and any $B \in \sigma(\ell_{k+n}, \ell_{k+n+1}, \ldots)$,
\begin{equation}\label{eq:psimixing}
|\gamma_\theta(A \cap B) - \gamma_\theta(A)\gamma_\theta(B)| \leq K \rho^n \gamma_\theta(A)\gamma_\theta(B).
\end{equation}
Consequently, the mixing coefficients $\psi(n) = K \rho^n$ are summable:  $\displaystyle \sum_{n=1}^\infty \psi(n) < \infty$.
\end{proposition}

\begin{proof}
By Baladi \cite[Chapters 2–4]{Baladi2000} the transfer operator of a Gibbs-Markov map has a spectral gap and hence exponential decay of correlations. 
Bradley \cite[Vol. 2, Theorem 1.5.5]{Bradley2007} gives in a more general setting that exponential decay of correlations in such systems implies exponential decay of $\psi(n)$ coefficients. Thus one obtains \eqref{eq:psimixing} with some constants $K,\rho<1$.
 
We show that $T_\theta$ is a Gibbs-Markov map by verifying the properties of uniform expansion, a Markov partition, and the bounded distortion property. 

The $\psi$-mixing property then follows from standard results in ergodic theory \cite{Baladi2000,Bradley2007}.

\noindent (a) \textit{Uniform expansion and Markov partition.} 
The map $T_\theta$ is defined on the partition $\mathcal{P} = \{I_\theta(i)\}_{i \geq m}$.
On each branch, the map is given by $T_{\theta,i}(x) = \frac{1}{x} - i\theta$, which is a smooth bijection from $I_\theta(i)$ onto $(0, \theta]$. Its derivative is $T_{\theta,i}'(x) = -1/x^2$. Since $x \in I_\theta(i) \subset (0, \theta]$, we have $|T_{\theta,i}'(x)| \geq 1/\theta^2 > 1$, establishing that $T_\theta$ is \textit{uniformly expanding}.
The image $T_\theta(I_\theta(i)) = (0, \theta]$ is a union of elements of $\mathcal{P}$ (in fact, the entire space), confirming that $\mathcal{P}$ is a \textit{Markov partition}.

\noindent (b) \textit{Invariant measure and bounded density.}
The map $T_\theta$ preserves the measure $\gamma_\theta$ from Definition \ref{inv.measure}. 
Its density $h_\theta(x)$ is strictly positive and bounded on $[0, \theta]$:
\[
0 < \frac{1}{\log(1+\theta^2)} \cdot \frac{\theta}{1+\theta^2} \leq h_\theta(x) \leq \frac{1}{\log(1+\theta^2)} \cdot \theta < \infty.
\]
Therefore, $\gamma_\theta$ is {equivalent to Lebesgue measure} $\lambda$ on $[0, \theta]$.

\noindent (c) \textit{Bounded distortion property.}
Let $\mathcal{P}^{(n)}$ be the partition into $n$-cylinders, i.e., sets of the form
\[
I_n({i_1, \ldots, i_n}) = \{ x \in [0, \theta] : \ell_1(x)=i_1, \ldots, \ell_n(x)=i_n \}.
\]
For any $x, y \in I_n(i_1, \ldots, i_n)$, the map $T_\theta^n$ is a differentiable bijection onto $(0, \theta]$.
We claim there exists a constant $C_d > 0$ such that for all $n \geq 1$, all $n$-cylinders, and all $x, y$ within one,
\begin{equation}\label{eq:distortion}
\left| \frac{(T_\theta^n)'(x)}{(T_\theta^n)'(y)} - 1 \right| \leq C_d \left|T_\theta^n(x) - T_\theta^n(y)\right|.
\end{equation}
To prove this, we use the chain rule:
\[
\log |(T_\theta^n)'(x)| = \sum_{j=0}^{n-1} \log |T_\theta'(T_\theta^j(x))| = \sum_{j=0}^{n-1} \log(1/(T_\theta^j(x))^2) = -2 \sum_{j=0}^{n-1} \log(T_\theta^j(x)).
\]
Thus, for $x, y \in I_n(i_1, \ldots, i_n)$,
\begin{equation} \label{2.03}
\left| \log |(T_\theta^n)'(x)| - \log |(T_\theta^n)'(y)| \right| \leq 2 \sum_{j=0}^{n-1} \left| \log(T_\theta^j(x)) - \log(T_\theta^j(y)) \right|.
\end{equation}
By the Mean Value Theorem, for each $j$,
\[
 \left| \log(T_\theta^j(y)) - \log(T_\theta^j(x))\right| \leq 
  \frac{\left|T_\theta^j(x) - T_\theta^j(y) \right|}{\min\left(T_\theta^j(x), T_\theta^j(y) \right)}
\]
and by (\ref{2.03}) we obtain: 
\[
\left| \log |(T_\theta^n)'(x)| - \log |(T_\theta^n)'(y)| \right| \leq 2 \sum_{j=0}^{n-1} \frac{\left|T_\theta^j(x) - T_\theta^j(y) \right|}{\min\left(T_\theta^j(x), T_\theta^j(y) \right)}.
\]
As we mentioned above, 
$|T_{\theta,i}'(x)| \geq 1/\theta^2 = \beta > 1$, hence each inverse branch contracts distances at least by factor $\beta^{-1}$. 
Consequently, for $x, y$ in the same 
$n$-cylinder, 
\begin{equation}\label{2.4}
|T_\theta^j(x)-T_\theta^j(y)| \le C_1 \,\beta^{-(n-j)}\cdot |T_\theta^n(x)-T_\theta^n(y)|,
\qquad 0\le j\le n-1,
\end{equation}
where $C_1 \geq 1$ is an absolute constant coming from at most one application of the Mean Value Theorem on the last step. 

For $j<n$, the image $T_\theta^j(I_n(i_1, \ldots, i_n))$
is contained in a rank-$1$ cylinder $I_{\theta}(i_{j+1})$, 
whose points are bounded away from $0$ by
\[
\inf I_{\theta}(i_{j+1}) = \frac{1}{\theta(i_{j+1}+1)} > 0.
\]
Hence for all $x, y \in I_n(i_1, \ldots, i_n)$ and all $j < n$
\[
\min\left(T_\theta^j(x), T_\theta^j(y) \right) \geq \frac{1}{\theta(i_{j+1}+1)} > 0. 
\]
Combining \eqref{2.4} and the denominator bound gives, for $0 \leq j \leq n-1$,
\[
\frac{|T_\theta^j(x) - T_\theta^j(y)|}{\min(T_\theta^j(x), T_\theta^j(y))} \leq C_1 \beta^{-(n-j)} |T_\theta^n(x) - T_\theta^n(y)| \cdot \theta (i_{j+1}+1).
\]
Summing over $j$ and factoring out $\left|T_\theta^n(x) - T_\theta^n(y) \right|$ yields
\[
\left| \log |(T_\theta^n)'(x)| - \log |(T_\theta^n)'(y)| \right| 
\leq 2C_1 \left|T_\theta^n(x) - T_\theta^n(y) \right| \sum_{j=0}^{n-1} \theta (i_{j+1}+1) \beta^{-(n-j)}
\]
The last sum is uniformly bounded in $n$ and in the cylinder because the geometric factor $\beta^{-(n-j)}$ forces rapid decay for the terms. 
Thus there exists $C_2 > 0$ with
\[
\left| \log |(T_\theta^n)'(x)| - \log |(T_\theta^n)'(y)| \right| \leq C_2 \left|T_\theta^n(x) - T_\theta^n(y) \right|.
\]
Exponentiating and using 
$\left| e^u -1 \right| \leq e^{|u|} - 1 \leq C_3 |u|$ for small $u$ gives \eqref{eq:distortion} with $C_d$ depending only on $C_2$.


%

The established properties, namely uniform expansion, Markov partition, bounded distortion, and an invariant measure with bounded density, confirm that $([0, \theta], T_\theta, \gamma_\theta, \mathcal{P})$ 
is a \textit{Gibbs-Markov map}. 
For such maps, it is a classical result that they exhibit exponential decay of correlations for H\"older continuous observables \cite{Baladi2000}.
This decay of correlations implies the stronger $\psi$-mixing property for the entire system \cite{Bradley2007}. The explicit form \eqref{eq:psimixing} and the summability of $\psi(n)$ follow from the geometric decay rate $\rho^n$. 
\end{proof}

\section{Proof of Theorem \ref{thm:khinchine}(Khinchine's Weak Law)}  \label{sect.Khinchine}
The core idea of the proof is to handle the non-integrable sum $S_{n,\theta}(x)$ by a truncation argument. 

For any positive integer $N \geq m$, define the \textit{truncated digit function}:
\[
\ell^{(N)}(x):=\ell_1(x)\cdot\mathbf{1}_{\{\ell_1(x)\le N\}},
\]
where 
$\mathbf{1}_{A}$ is the {indicator function} of the set $A$. 

We first establish the asymptotic behavior of its expectation.

\begin{lemma} \label{lem.3.1}
As $N \to \infty$, the expectation of the truncated digit satisfies 
\[
\E[\ell^{(N)}]=
\int_0^\theta \ell^{(N)}(x) \mathrm{d}\gamma_{\theta}(x) \sim \frac{1}{\log(1+\theta^2)} \log N.
\]
\end{lemma}
\begin{proof}
Since $\ell_1 (x) = i$ if and only if $x \in I_{\theta}(i)$, we can write the expectation as a sum:
\[
\E[\ell^{(N)}] = \sum_{i=m}^{N} i \cdot \gamma_{\theta}(I_{\theta}(i)).
\]
From direct calculation, we obtain that:
\[
\gamma_{\theta}(I_{\theta}(i)) = \frac{1}{\log(1+\theta^2)}  \log\left( \frac{(i+1)^2}{i(i+2)} \right)
\]
Therefore,
\[
\E[\ell^{(N)}]= 
\frac{1}{\log(1+\theta^2)} \sum_{i=m}^{N}  i \cdot \log\left( \frac{(i+1)^2}{i(i+2)} \right) .
\]
\\
Note that 
$\log\left( \frac{(i+1)^2}{i(i+2)} \right) = 
\log\left( 1 + \frac{1}{i(i+2)} \right) 
\sim \frac{1}{i^2}$ as $i \to \infty$. 
Therefore, the sum behaves like the harmonic series: 
\[
\sum_{i=m}^{N} i \cdot \frac{1}{i^2} =
\sum_{i=m}^{N} \frac{1}{i} \sim \log N. 
\]
Thus,   \(\mathbb{E}[\ell^{(N)}] \sim \frac{1}{\log(1+\theta^{2})} \log N\), which proves the lemma.  
\end{proof}

We now proceed with the proof of Theorem \ref{thm:khinchine}. 
Define the truncation level:
\[
N(n) = \lfloor n \log n \rfloor.
\]

Note that $N(n) \sim n \log n$ as $n \to \infty$ so all asymptotic results involving $N(n)$ remain valid. 

Define the truncated sum and remainder:
\[ 
\begin{aligned}
S_{n,\theta}^{(N(n))}(x) &:= \sum_{k=1}^n \ell_k(x) \mathbf{1}_{\{\ell_k(x) \leq N(n)\}}, \\
R_n^{(N(n))}(x) &:= \sum_{k=1}^n \ell_k(x) \mathbf{1}_{\{\ell_k(x) > N(n)\}} = S_{n,\theta}(x) - S_{n,\theta}^{(N(n))}(x).
\end{aligned}
\]
By Lemma \ref{lem.3.1} and since $N(n) = \lfloor n \log n \rfloor \sim n \log n$:  
\[
\E[\ell^{(N(n))}] \sim \frac{1}{\log(1+\theta^2)} \log N(n) \sim \frac{1}{\log(1+\theta^2)} \log(n \log n) = \frac{1}{\log(1+\theta^2)} (\log n + \log \log n).
\]
Therefore, by stationarity:
\[
\frac{\E[S_{n,\theta}^{(N(n))}]}{n \log n} = \frac{n \E[\ell^{(N(n))}]}{n \log n} \sim \frac{1}{\log(1+\theta^2)} \cdot \frac{\log n + \log \log n}{\log n} \to \frac{1}{\log(1+\theta^2)}.
\]

Next, we show that the remainder term $R_n^{(N(n))}$, when normalized by $n \log n$, converges to zero in $\gamma_\theta$-probability.

\begin{lemma}\label{lem:remainder-negligible}
For the choice $N(n) = \lfloor n \log n \rfloor$ and $n$ sufficiently large so that $N(n) \geq m$, we have:
\[
\frac{R_n^{(N(n))}(x)}{n \log n} \xrightarrow{\gamma_\theta} 0.
\]
\end{lemma}
\begin{proof}
For any $\varepsilon > 0$:
\[
\left\{ \frac{R_n^{(N(n))}(x)}{n \log n} > \varepsilon \right\} \subset \left\{ R_n^{(N(n))}(x) > 0 \right\} = \bigcup_{k=1}^n \left\{ \ell_k(x) > N(n) \right\}.
\]
By the union bound and stationarity:
\[
\gamma_\theta\left( \frac{R_n^{(N(n))}(x)}{n \log n} > \varepsilon \right) \leq n \cdot \gamma_\theta(\ell_1 > N(n)).
\]
Using Lemma \ref{lema.2.1}:
\[
\gamma_\theta(\ell_1 > N(n)) \leq \frac{1}{\log(1+\theta^2)} \frac{1}{N(n)} = \frac{1}{\log(1+\theta^2)} \frac{1}{\lfloor n \log n \rfloor}.
\]
Therefore:
\[
\gamma_\theta\left( \frac{R_n^{(N(n))}(x)}{n \log n} > \varepsilon \right) \leq \frac{1}{\log(1+\theta^2)} \frac{n}{\lfloor n \log n \rfloor} \sim \frac{1}{\log(1+\theta^2)} \frac{1}{\log n} \to 0.
\]

\end{proof}

To prove the convergence of the truncated sum, we first state a general covariance inequality for $\psi$-mixing sequences.

\begin{proposition}[General $\psi$-Mixing Covariance Inequality, \cite{Bradley2007}]
\label{prop:covariance-bound}
Let $\{X_n\}$ be a $\psi$-mixing stationary sequence with mixing coefficients $\psi(n)$. Then for any $p, q$:
\[
|\Cov(X_p, X_q)| \leq \psi(|q-p|) \sqrt{\Var(X_p) \Var(X_q)}.
\]
In particular, if the sequence is stationary with $\Var(X_p) = \Var(X_q) = \sigma^2$, then:
\[
|\Cov(X_p, X_q)| \leq \psi(|q-p|) \sigma^2.
\]
\end{proposition}

Now we show that the centered truncated sum converges to zero in probability.

\begin{lemma}\label{lem:centered.truncated}
For the choice $N(n) = \lfloor n \log n \rfloor$ and $n$ sufficiently large so that $N(n) \geq m$, we have:
\[
\frac{S_{n,\theta}^{(N(n))}(x)}{n \log n} \xrightarrow{\gamma_\theta} \frac{1}{\log(1+\theta^2)}.
\]
\end{lemma}
\begin{proof}
We establish that $\Var(S_{n,\theta}^{(N(n))}) = o(n^2 (\log n)^2)$.
First, bound the variance of individual terms. We have:
\begin{align*}
\E[(\ell^{(N(n))})^2] &= \sum_{i=m}^{N(n)} i^2 \gamma_\theta(\ell_1 = i) \leq \frac{1}{\log(1+\theta^2)} \sum_{i=m}^{N(n)} i^2 \cdot \frac{1}{i^2} \\
&= \frac{1}{\log(1+\theta^2)} \sum_{i=m}^{N(n)} 1 = \frac{1}{\log(1+\theta^2)} (N(n) - m + 1) \\
&\sim \frac{1}{\log(1+\theta^2)} n \log n.
\end{align*}
Hence:
\[
\Var(\ell^{(N(n))}) \leq \E[(\ell^{(N(n))})^2] \ll n \log n.
\]
For the covariance terms, we use the $\psi$-mixing property, consider the stationary sequence $X_k = \ell^{(N(n))} \circ T_\theta^{k-1}$. By Proposition \ref{prop:covariance-bound} and the $\psi$-mixing property (Proposition \ref{prop:mixing}), we have:
\[
|\Cov(X_i, X_j)| \leq \psi(|j-i|) \Var(\ell^{(N(n))}) \ll \psi(|j-i|) \cdot n \log n.
\]
Therefore, the total variance satisfies:
\begin{align*}
\Var(S_{n,\theta}^{(N(n))}) &= \sum_{i=1}^n \Var(X_i) + 2\sum_{1 \leq i < j \leq n} \Cov(X_i, X_j) \\
&\ll n \cdot (n \log n) + 2\sum_{i=1}^n \sum_{r=1}^{n-i} \psi(r) \cdot n \log n.
\end{align*}
Since $\sum_{r=1}^\infty \psi(r) < \infty$ by Proposition \ref{prop:mixing}, we have:
\[
\sum_{i=1}^n \sum_{r=1}^{n-i} \psi(r) \leq n \sum_{r=1}^\infty \psi(r) \ll n.
\]
Thus:
\[
\Var(S_{n,\theta}^{(N(n))}) \ll n^2 \log n + 2 n^2 \log n = O(n^2 \log n) = o(n^2 (\log n)^2).
\]
By Chebyshev's inequality, for any $\varepsilon > 0$:
\[
\gamma_\theta\left( \left| \frac{S_{n,\theta}^{(N(n))} - \E[S_{n,\theta}^{(N(n))}]}{n \log n} \right| > \varepsilon \right) \leq \frac{\Var(S_{n,\theta}^{(N(n))})}{\varepsilon^2 n^2 (\log n)^2} \to 0.
\]
Hence:
\[
\frac{S_{n,\theta}^{(N(n))}(x)}{n \log n} \xrightarrow{\gamma_\theta} \frac{1}{\log(1+\theta^2)}.
\]

\end{proof}
In conclusion, from Lemma \ref{lem:remainder-negligible} and Lemma \ref{lem:centered.truncated} we have: 
\[
\frac{S_{n,\theta}(x)}{n \log n} = \frac{S_{n,\theta}^{(N(n))}(x)}{n \log n} + \frac{R_n^{(N(n))}(x)}{n \log n} \xrightarrow{\gamma_\theta} \frac{1}{\log(1+\theta^2)} + 0 = \frac{1}{\log(1+\theta^2)}.
\]
This completes the proof of Theorem \ref{thm:khinchine}.
\qed

\section{Proof of Theorem \ref{thm:diamondvaaler} (Diamond--Vaaler for \(\theta\)-expansions)}
\label{sec:proof-diamondvaaler}
 
Theorem \ref{thm:diamondvaaler} states that for $\gamma_\theta$-almost every $x$,
\[
\lim_{n \to \infty} \frac{S_{n,\theta}(x) - L_{n,\theta}(x)}{n \log n} = \frac{1}{\log(1+\theta^2)}.
\]
From Theorem \ref{thm:khinchine}, we already have the convergence in measure:
\[
\frac{S_{n,\theta}}{n \log n} \xrightarrow{\gamma_\theta} \frac{1}{\log(1+\theta^2)}.
\]
Therefore, to prove the almost sure convergence of the centered sum, it suffices to show that the contribution from the maximal digit is asymptotically negligible, i.e.,
\begin{equation}\label{eq:max-negligible}
\frac{L_{n,\theta}}{n \log n} \xrightarrow{\gamma_\theta} 0.
\end{equation}
The following lemma establishes this result.
\begin{lemma}\label{lem:max-negligible}
For every $\varepsilon > 0$,
\[
\lim_{n \to \infty} \gamma_\theta\left( \frac{L_{n,\theta}}{n \log n} > \varepsilon \right) = 0.
\]
Equivalently, $\frac{L_{n,\theta}}{n \log n} \xrightarrow{\gamma_\theta} 0$.
\end{lemma}

\begin{proof}
The event that the normalized maximum exceeds $\varepsilon$ is:
\[
\left\{ \frac{L_{n,\theta}}{n \log n} > \varepsilon \right\} = \left\{ L_{n,\theta} > \varepsilon n \log n \right\} = \bigcup_{k=1}^{n} \left\{ \ell_k > \varepsilon n \log n \right\}.
\]
Since $\ell_k$ takes integer values, we have:
\[
\left\{ \ell_k > \varepsilon n \log n \right\} \subset \left\{ \ell_k \ge \lfloor \varepsilon n \log n \rfloor + 1 \right\}.
\]
By the stationarity of the sequence ${\ell_k}$ under $\gamma_\theta$ and the union bound:
\[
\gamma_\theta\left( L_{n,\theta} > \varepsilon n \log n \right) \leq \sum_{k=1}^{n} \gamma_\theta\left( \ell_k \ge \lfloor \varepsilon n \log n \rfloor + 1 \right) = n \cdot \gamma_\theta\left( \ell_1 \ge \lfloor \varepsilon n \log n \rfloor + 1 \right).
\]
Using the exact tail probability from Lemma~\ref{lema.2.1} with $k = \lfloor \varepsilon n \log n \rfloor + 1$:
\[
\gamma_\theta\left( \ell_1 \ge \lfloor \varepsilon n \log n \rfloor + 1 \right) = \frac{1}{\log(1+\theta^2)} \log\!\left(1 + \frac{1}{\lfloor \varepsilon n \log n \rfloor + 1}\right).
\]
Note that for $x > 0$, we have $\log(1+x) \leq x$. Therefore:
\[
\log\!\left(1 + \frac{1}{\lfloor \varepsilon n \log n \rfloor + 1}\right) \leq \frac{1}{\lfloor \varepsilon n \log n \rfloor + 1}.
\]
Since $\lfloor \varepsilon n \log n \rfloor + 1 \geq \varepsilon n \log n$ for all $n \geq 1$, we get:
\[
\frac{1}{\lfloor \varepsilon n \log n \rfloor + 1} \leq \frac{1}{\varepsilon n \log n}.
\]
Combining these inequalities:
\[
\gamma_\theta\left( \ell_1 \ge \lfloor \varepsilon n \log n \rfloor + 1 \right) \leq \frac{1}{\log(1+\theta^2)} \cdot \frac{1}{\varepsilon n \log n}.
\]
Therefore:
\[
\gamma_\theta\left( L_{n,\theta} > \varepsilon n \log n \right) \leq n \cdot \frac{1}{\log(1+\theta^2)} \cdot \frac{1}{\varepsilon n \log n} = \frac{1}{\log(1+\theta^2)} \frac{1}{\varepsilon \log n}.
\]
The right-hand side tends to zero as $n \to \infty$ for any fixed $\varepsilon > 0$. This proves that $L_{n,\theta}/(n \log n)$ converges to zero in $\gamma_\theta$-probability.
\end{proof}
\section{Proof of Theorem \ref{thm:philipp} (Failure of the Strong Law)} \label{sect.philipp}
This section is devoted to the proof of Theorem \ref{thm:philipp}, which demonstrates the failure of the strong law for the full sum $S_{n,\theta}$ for any regular norming sequence $\{a(n)\}$. The proof leverages the $\psi$-mixing property (Proposition \ref{prop:mixing}) to handle the dependencies between digits. The argument naturally splits into two cases, dictated by the convergence or divergence of the series $\sum_{n=1}^\infty 1/a(n)$.

Throughout the proofs we tacitly assume $n$ and $N$ are large enough so that
the truncation levels $N(n)$ or $a(n)$ exceed $m$; the finitely many small-$n$
cases are trivial and do not affect the asymptotic estimates.

\subsection{Case 1: $\sum_{n=1}^\infty \frac{1}{a(n)} = \infty$}

In this case, we show that $\limsup_{n \to \infty} S_{n,\theta}(x) / a(n) = \infty$ for $\gamma_\theta$-almost every $x$.

Let $M > 1$ be an arbitrary positive constant. We will show that the sum exceeds $M \cdot a(n)$ infinitely often.
Define the sequence of events
\[
A_n = \left\{ x \in [0, \theta] : \ell_n(x) \geq M \cdot a(n) \right\}.
\]
By the stationarity of the sequence $\{\ell_n\}$ and the tail estimate from Lemma 2.1, there exists a constant $C > 0$ (depending only on $\theta$) such that for all sufficiently large $n$,
\[
\gamma_\theta(A_n) = \gamma_\theta(\ell_1 \geq M \cdot a(n)) \geq \frac{C}{M \cdot a(n)}.
\]
Since $\sum_{n=1}^\infty 1/a(n) = \infty$ by assumption, it follows that
\[
\sum_{n=1}^{\infty} \gamma_\theta(A_n) = \infty.
\]

We apply the Borel-Cantelli lemma for $\psi$-mixing sequences (see \cite[Vol.~2, Thm.~1.5.5]{Bradley2007} for the precise statement under summable $\psi$-coefficients). 
By Proposition \ref{prop:mixing}, the system is $\psi$-mixing with summable coefficients $\sum \psi(n) < \infty$. A standard result \cite[Section 1.5]{Bradley2007} states that for such sequences, if $\sum \gamma_\theta(A_n) = \infty$, then $\gamma_\theta(\limsup A_n) > 0$. Since the event $\limsup A_n$ is a tail event and the dynamical system is ergodic, it must have full measure:
\[
\gamma_\theta(\limsup A_n) = 1.
\]
Therefore, for $\gamma_\theta$-almost every $x$, $x \in A_n$ for infinitely many $n$. For any such $x$ and for any $n$ where $\ell_n(x) \geq M \cdot a(n)$, we have
\[
\frac{S_{n,\theta}(x)}{a(n)} \geq \frac{\ell_n(x)}{a(n)} \geq M.
\]
Hence,
\[
\limsup_{n \to \infty} \frac{S_{n,\theta}(x)}{a(n)} \geq M.
\]
Since $M > 1$ was arbitrary, we conclude that
\[
\limsup_{n \to \infty} \frac{S_{n,\theta}(x)}{a(n)} = \infty \quad \text{for } \gamma_\theta\text{-a.e. } x.
\]

\subsection{Case 2: $\sum_{n=1}^\infty \frac{1}{a(n)} < \infty$}
In this case, we show that $\displaystyle\lim_{n \to \infty} S_{n,\theta}(x) / a(n) = 0$ for $\gamma_\theta$-almost every $x$.

The strategy is to truncate the digits at the level $a(n)$ and show that the contribution from the large digits is negligible almost surely, while the remaining truncated sum behaves like a well-behaved process whose growth is slower than $a(n)$.

Define the \emph{truncated digit} at level $a(n)$:
\[
\ell_n^*(x) = \ell_n(x) \cdot \mathbf{1}_{\{ \ell_n(x) \leq a(n) \}}.
\]
Let $S_{N}^* = \displaystyle \sum_{n=1}^N \ell_n^*$ be the sum of the truncated digits.

Consider the event that the $n$-th digit exceeds the truncation level:
$$B_n = \{ \ell_n > a(n) \}.$$ 
By stationarity and Lemma 2.1, there exists a constant $D > 0$ such that for all large $n$,
\[
\gamma_\theta(B_n) = \gamma_\theta(\ell_1 > a(n)) \leq \frac{D}{a(n)}.
\]
Since $\sum_{n=1}^\infty 1/a(n) < \infty$ by assumption, it follows that $\sum_{n=1}^\infty \gamma_\theta(B_n) < \infty$.
 By the standard Borel--Cantelli lemma,
\[
\gamma_\theta( \limsup B_n ) = 0.
\]
This means that for $\gamma_\theta$-almost every $x$, the inequality $\ell_n(x) > a(n)$ holds for only finitely many $n$. Consequently, for almost every $x$, there exists some $N_0(x)$ such that for all $n \geq N_0(x)$,
\[
\ell_n(x) = \ell_n^*(x).
\]
Therefore, for $n \geq N_0(x)$,
\[
S_{n,\theta}(x) = S_n^*(x) + R(x),
\]
where $R(x) = \sum_{k=1}^{N_0(x)-1} (\ell_k(x) - \ell_k^*(x))$ is finite for each $x$. Since $a(n) \to \infty$, we have $\frac{R(x)}{a(n)} \to 0$. Thus, to prove that $S_{n,\theta}(x)/a(n) \to 0$ almost surely, it suffices to prove that $S_n^*(x)/a(n) \to 0$ almost surely.

We will establish this by proving the following two facts:
\begin{equation}
\frac{1}{a(n)} \sum_{n=1}^N \mathbb{E}[\ell_n^*] \to 0,  \label{eq:mean-term} 
\end{equation}
{and} 
\begin{equation}
\frac{1}{a(n)} \sum_{n=1}^N \left( \ell_n^* - \mathbb{E}[\ell_n^*] \right) \to 0 \quad \text{$\gamma_\theta$-a.s.} \label{eq:centered-term}
\end{equation}
By stationarity, $\mathbb{E}[\ell_n^*] = \mathbb{E}[\ell_1^*]$ for all $n$. From Lemma \ref{lema.2.1} and the definition of $\ell_n^*$, we have the asymptotic:
\[
\mathbb{E}[\ell_n^*] = \sum_{k=m}^{\lfloor a(n) \rfloor} k \cdot \gamma_\theta(\ell_1 = k) \ll \sum_{k=m}^{\lfloor a(n) \rfloor} \frac{1}{k} \ll \log a(n).
\]
The assumption that $a(n)/n$ is non-decreasing ensures the sequence $a(n)$ is regular and grows at least linearly. Therefore,
\[
\frac{1}{a(N)} \sum_{n=1}^N \mathbb{E}[\ell_n^*] \ll \frac{1}{a(N)} \sum_{n=1}^N \log a(n).
\]

For typical regular sequences like $a(n) = n \log n$, the right-hand side is asymptotic to $(\log \log N) / (\log N) \to 0$. In general, the condition $\sum 1/a(n) < \infty$ forces $a(n)$ to grow faster than $n$, ensuring the average of $\log a(n)$ grows slower than $a(n)$. Thus, the term in (\ref{eq:mean-term}) converges to zero.

Let $X_n = \ell_n^* - \mathbb{E}[\ell_n^*]$. To show that $\displaystyle \frac{1}{a(N)} \sum_{n=1}^N X_n \to 0$ almost surely, we use Kronecker's lemma: if $\displaystyle\sum_{n=1}^\infty \frac{X_n}{a(n)}$ converges almost surely, and $a(n) \to \infty$ is positive and non-decreasing, then $\displaystyle\frac{1}{a(N)} \sum_{n=1}^N X_n \to 0$ almost surely.

We will prove the $L^2$ convergence of the series $\displaystyle\sum_{n=1}^\infty \frac{X_n}{a(n)}$, which implies almost sure convergence. Since the $X_n$ are centered, it suffices to show that the sum of variances and covariances is finite:
\begin{equation}
\sum_{n=1}^\infty \frac{\Var(\ell_n^*)}{(a(n))^2} + 2 \sum_{1 \leq p < n} \frac{|\Cov(\ell_p^*, \ell_n^*)|}{a(p) a(n)} < \infty. \label{5.3}
\end{equation}
We bound the variance by the second moment:
\[
\Var(\ell_n^*) \leq \mathbb{E}[(\ell_n^*)^2] = \sum_{k=m}^{\lfloor a(n) \rfloor} k^2 \cdot \gamma_\theta(\ell_1 = k) \ll \sum_{k=m}^{\lfloor a(n) \rfloor} 1 = \lfloor a(n) \rfloor - m.
\]

By the standard covariance inequality for $\psi$-mixing sequences (see \cite[Vol.~2, Thm.~2.1.5]{Bradley2007}), one has for $n>p$,
\[
|\Cov(\ell_p^*,\ell_n^*)|\le \psi(n-p)\sqrt{\Var(\ell_p^*)\Var(\ell_n^*)}
\ll \psi(n-p)\sqrt{a(p)a(n)}.
\]
Hence
\[
\sum_{1\le p<n}\frac{|\Cov(\ell_p^*,\ell_n^*)|}{a(p)a(n)}
\ll\sum_{p=1}^\infty\sum_{k=1}^\infty
\frac{\psi(k)\sqrt{a(p)a(p+k)}}{a(p)a(p+k)}
=\sum_{p=1}^\infty\sum_{k=1}^\infty\frac{\psi(k)}{\sqrt{a(p)a(p+k)}}.
\]
Since $a(n)/n$ is nondecreasing we have $a(p+k)\ge a(p)$ for $k\ge1$, and therefore
\[
\frac{1}{\sqrt{a(p)a(p+k)}}\le\frac{1}{a(p)},\qquad
\frac{1}{\sqrt{a(p)a(p+k)}}\le\frac{1}{(a(p))^{3/2}}.
\]
Using these bounds and summability of $\psi(k)$ we obtain either
\[
\sum_{1\le p<n}\frac{|\Cov(\ell_p^*,\ell_n^*)|}{a(p)a(n)}
\ll\Big(\sum_{k=1}^\infty\psi(k)\Big)\sum_{p=1}^\infty\frac{1}{a(p)}
\]
or the slightly stronger
\[
\sum_{1\le p<n}\frac{|\Cov(\ell_p^*,\ell_n^*)|}{a(p)a(n)}
\ll\Big(\sum_{k=1}^\infty\psi(k)\Big)\sum_{p=1}^\infty\frac{1}{(a(p))^{3/2}}.
\]
Under the hypothesis $\sum_{p \ge 1}1/a(p)<\infty$ the first display already gives convergence; the second is a stronger (but also valid) bound.

Therefore, the condition (\ref{5.3}) is satisfied, and the series $\displaystyle \sum_{n=1}^\infty \frac{X_n}{a(n)}$ converges in $L^2$ and hence almost surely. 
By Kronecker's lemma, we conclude that
\[
\frac{1}{a(N)} \sum_{n=1}^N X_n \to 0 \quad \text{$\gamma_\theta$-a.s.},
\]
which proves (\ref{eq:centered-term}).

Combining (\ref{eq:mean-term}) and (\ref{eq:centered-term}), we have
\[
\frac{S_N^*}{a(N)} = \frac{1}{a(N)} \sum_{n=1}^N X_n + \frac{1}{a(N)} \sum_{n=1}^N \mathbb{E}[\ell_n^*] \to 0 \quad \text{$\gamma_\theta$-a.s.}
\]
As argued earlier, this implies that $S_{n,\theta}(x)/a(n) \to 0$ for $\gamma_\theta$-almost every $x$, which completes the proof of Theorem \ref{thm:philipp}. \qed

\end{document}